\documentclass[a4paper,12pt]{extarticle}

\usepackage[utf8]{inputenc}
\usepackage{geometry}
\usepackage{graphicx}
\usepackage{booktabs}
\usepackage{mathrsfs}
\usepackage{bm}
\usepackage{xcolor}
\usepackage{algorithm}
\usepackage{algpseudocode}
\usepackage{mathtools}
\usepackage{enumitem}
\usepackage{tikz-cd}

\usepackage{newtxtext}
\usepackage{newtxmath}

\usepackage{amssymb}

\usepackage{amsthm}

\usepackage{hyperref}
\usepackage{xcolor} 

\newtheorem{theorem}{Theorem}
\newtheorem{proposition}[theorem]{Proposition}
\newtheorem{lemma}[theorem]{Lemma}
\newtheorem{corollary}[theorem]{Corollary}
\newtheorem{definition}{Definition}

\newcommand{\ch}{\operatorname{ch}}
\newcommand{\td}{\operatorname{td}}

\theoremstyle{definition}

\newtheorem{remark}{Remark}

\begin{document}

\pagestyle{plain}

\title{Spencer-Riemann-Roch Theory: Mirror Symmetry of Hodge Decompositions and Characteristic Classes in Constrained Geometry}
\author{Dongzhe Zheng\thanks{Department of Mechanical and Aerospace Engineering, Princeton University\\ Email: \href{mailto:dz5992@princeton.edu}{dz5992@princeton.edu}, \href{mailto:dz1011@wildcats.unh.edu}{dz1011@wildcats.unh.edu}}}
\date{}
\maketitle

\begin{abstract}
The discovery of mirror symmetry in compatible pair Spencer complex theory brings new theoretical tools to the study of constrained geometry. Inspired by classical Spencer theory and modern Hodge theory, this paper establishes Spencer-Riemann-Roch theory in the context of constrained geometry, systematically studying the mirror symmetry of Spencer-Hodge decompositions and their manifestations in algebraic geometry. We utilize Serre's GAGA principle to algebraic geometrize Spencer complexes, establish coherent sheaf formulations, and reveal the topological essence of mirror symmetry through characteristic class theory. Main results include: Riemann-Roch type Euler characteristic computation formulas for Spencer complexes, equivalence theorems for mirror symmetry of Hodge decompositions at the characteristic class level, and verification of these theories in concrete geometric constructions. Research shows that algebraic geometric methods can not only reproduce deep results from differential geometry, but also reveal the intrinsic structure of mirror symmetry in constrained geometry through characteristic class analysis, opening new directions for applications of Spencer theory in constrained geometry.
\end{abstract}

\section{Introduction}

Spencer complexes, as important tools in differential geometry for handling overdetermined differential equation systems, have theoretical foundations traceable to Spencer's pioneering work\cite{spencer1962deformation}. With further development by Goldschmidt\cite{goldschmidt1967integrability} and Quillen\cite{quillen1969formal}, Spencer theory gradually formed a complete mathematical framework. In recent years, the development of constrained geometry has injected new vitality into Spencer theory, particularly the introduction of compatible pair concepts, which provides a unified geometric framework for handling constrained dynamical systems.

The systematic development of compatible pair Spencer complex theory is based on the solid foundation of a series of preliminary works\cite{zheng2025dynamical,zheng2025geometric,zheng2025constructing,zheng2025mirror}. Reference\cite{zheng2025dynamical} established the dynamical geometric theory of principal bundle constrained systems through the pairing of constraint distribution $D$ and dual constraint function $\lambda$, with the introduction of strong transversality conditions and modified Cartan equations forming the core framework of compatible pair theory. Reference\cite{zheng2025geometric} discovered the deep symmetry of mirror transformation $(D,\lambda) \mapsto (D,-\lambda)$, establishing geometric duality theory between constraints and gauge fields. Reference\cite{zheng2025constructing} established complete Hodge decomposition theory by constructing Spencer metrics, providing a solid analytical foundation for Spencer complexes. Reference\cite{zheng2025mirror} further revealed the mirror symmetry of Spencer-Hodge decompositions, discovering metric invariance and cohomological equivalence, providing clear targets for this paper's algebraic geometric advancement.

Traditional Spencer theory mainly focuses on local integrability problems, while the innovation of compatible pair theory lies in establishing global geometric structures through the pairing of constraint distributions and dual constraint functions. This approach is deeply influenced by Cartan geometry\cite{sharpe1997differential} and modern gauge field theory, extending the analysis of constrained systems from local to global. The introduction of strong transversality conditions ensures the ellipticity of Spencer complexes, which lays the foundation for subsequent Hodge theory analysis.

Hodge theory, as a core tool of modern geometric analysis\cite{warner1983foundations}, demonstrates the deep connection between constrained geometry and elliptic theory through its application in Spencer complexes. The establishment of compatible pair Spencer-Hodge decomposition not only proves the finite dimensionality of harmonic spaces, but more importantly, reveals topological invariants in constrained geometry. This decomposition is similar in spirit to the Atiyah-Singer index theorem\cite{atiyah1963index}'s treatment of elliptic operators, but presents unique characteristics in the context of constrained geometry.

The discovery of mirror symmetry brings unexpected depth to compatible pair theory. The invariance of Spencer structures under the transformation $(D,\lambda) \mapsto (D,-\lambda)$ is not only a technical result, but also reveals the intrinsic symmetry of constrained geometry. This symmetry echoes mirror symmetry phenomena in mathematical physics\cite{strominger1996mirror} to some extent, although their mathematical mechanisms are completely different. The mirror invariance of Spencer metrics and the symmetry of cohomology groups indicate that compatible pair theory may have more profound mathematical structure than appears on the surface.

The introduction of algebraic geometric methods provides new possibilities for deepening this theoretical understanding. Inspired by the algebraic geometric revolution of the Grothendieck school\cite{grothendieck1957theoremes}, we recognize that many differential geometric problems can gain deeper understanding within the algebraic geometric framework. Serre's GAGA principle\cite{serre1956geometrie} provides fundamental tools for connecting analytic geometry with algebraic geometry, enabling us to transform the analysis of Spencer complexes into the study of coherent sheaves.

The research motivation of this paper stems from the need for further understanding of the mathematical essence of compatible pair theory. While differential geometric methods provide intuitive geometric pictures for theory establishment, the abstractness of algebraic geometry may reveal deeper structures of problems. In particular, the manifestation of mirror symmetry at the characteristic class level may have more essential significance than metric invariance. The Hirzebruch-Riemann-Roch theorem\cite{hirzebruch1966topological}, as a bridge connecting topology and geometry, provides new approaches for Euler characteristic computation of Spencer complexes.

We must acknowledge that this research still faces significant limitations. The strict Lie group conditions (compact connected semi-simple with trivial center), while ensuring mathematical rigor of the theory, also limit the scope of applications. Many practically important Lie groups do not satisfy these conditions, which constrains direct applications of the theory. Nevertheless, we believe that the theoretical framework established under strict conditions provides a reliable foundation for subsequent generalizations. Strict assumptions are adopted solely for foundational rigor; extensions beyond these constraints constitute future directions.

The main contributions of this paper can be summarized as four aspects of preliminary attempts. First, we establish the algebraic geometric formulation of Spencer complexes, reinterpreting the basic concepts of compatible pairs using coherent sheaf theory. Second, we prove the equivalence of mirror symmetry at the characteristic class level, transforming metric invariance into topological properties. Third, we derive Euler characteristic computation formulas based on Riemann-Roch theory, providing new tools for quantitative analysis of Spencer complexes. Finally, we verify the internal consistency of the theory through concrete calculations of $\text{PSU}(2)$ on the complex projective plane.

We hope this preliminary exploration can provide some beneficial insights for cross-research between constrained geometry and algebraic geometry. Although our results are still preliminary and the applicable scope of the theory is limited, the potential demonstrated by algebraic geometric methods is encouraging. We believe that with further development of techniques and gradual relaxation of conditions, this cross-field area may produce more profound mathematical understanding.

\section{Compatible Pair Spencer Theory in Algebraic Geometric Context}

This section establishes the theoretical foundation for transforming compatible pair Spencer theory from differential geometry to algebraic geometry, focusing on reviewing core concepts and results needed for algebraic geometrization. We will establish strict correspondences through the GAGA principle and ensure preservation of the deep structures of the original theory within the algebraic geometric framework.

\subsection{Differential Geometric Foundation of Compatible Pair Spencer Complexes}

The compatible pair theory established in reference\cite{zheng2025dynamical} provides a solid starting point for our algebraic geometrization. Compatible pairs $(D,\lambda)$ consist of constraint distribution $D \subset TP$ and dual constraint function $\lambda: P \to \mathfrak{g}^*$, achieving unified description of constrained geometry through three core conditions. The \textbf{strong transversality condition} $D_p \oplus V_p = T_pP$ ensures geometric non-degeneracy of the constraint distribution, where $V_p$ is the vertical subspace of the fiber. The \textbf{modified Cartan equation}
\begin{equation}\label{eq:cartan}
d\lambda + \mathrm{ad}^*_\omega \lambda = 0
\end{equation}
establishes the covariant relationship between dual functions and principal bundle connections, embodying the deep coupling between constraints and gauge fields. The \textbf{compatibility condition}
\begin{equation}\label{eq:compatibility}
D_p = \{v \in T_pP : \langle\lambda(p), \omega(v)\rangle = 0\}
\end{equation}
unifies algebraic constraint conditions with geometric distributions.

Based on compatible pair structures, Spencer complexes $(S^{\bullet}_{D,\lambda}, D^{\bullet}_{D,\lambda})$ can be constructed. Spencer spaces are defined as
\begin{equation}\label{eq:spencer-space}
S^k_{D,\lambda} = \Omega^k(M) \otimes \mathrm{Sym}^k(\mathfrak{g})
\end{equation}
combining exterior differential forms with symmetric tensors of the Lie algebra. Spencer differential operators are composed of standard exterior differentials and Spencer extension operators:
\begin{equation}\label{eq:spencer-diff}
D^k_{D,\lambda}(\omega \otimes s) = d\omega \otimes s + (-1)^k\omega \otimes \delta^{\lambda}_{\mathfrak{g}}(s)
\end{equation}



\begin{definition}[Constructive Spencer Extension Operator]
\label{def:constructive_spencer_operator}
The \textbf{Spencer extension operator} $\delta^{\lambda}_{\mathfrak{g}}$ is defined as a \textbf{graded derivation of degree +1} on the symmetric algebra $S(\mathfrak{g}) = \bigoplus_{k=0}^{\infty} \mathrm{Sym}^k(\mathfrak{g})$. It is uniquely determined by the following two rules:
\begin{enumerate}
    \item[\textbf{(A)}] \textbf{Action on Generators ($k=1$):} For any Lie algebra element $v \in \mathfrak{g} = \mathrm{Sym}^1(\mathfrak{g})$, its image $\delta^{\lambda}_{\mathfrak{g}}(v)$ is a second-order symmetric tensor in $\mathrm{Sym}^2(\mathfrak{g})$, defined by its action on two test vectors $w_1, w_2 \in \mathfrak{g}$:
    \begin{equation}
        (\delta^{\lambda}_{\mathfrak{g}}(v))(w_1, w_2) := \frac{1}{2} \left( \langle\lambda, [w_1, [w_2, v]]\rangle + \langle\lambda, [w_2, [w_1, v]]\rangle \right)
    \end{equation}
    where $\langle \cdot, \cdot \rangle$ denotes the pairing between $\mathfrak{g}^*$ and $\mathfrak{g}$. This construction ensures the output is symmetric in $w_1$ and $w_2$.

    \item[\textbf{(B)}] \textbf{Graded Leibniz Rule:} For any homogeneous elements $s_1 \in \mathrm{Sym}^p(\mathfrak{g})$ and $s_2 \in \mathrm{Sym}^q(\mathfrak{g})$, the operator satisfies:
    \begin{equation}
        \delta^{\lambda}_{\mathfrak{g}}(s_1 \odot s_2) := \delta^{\lambda}_{\mathfrak{g}}(s_1) \odot s_2 + (-1)^p s_1 \odot \delta^{\lambda}_{\mathfrak{g}}(s_2)
    \end{equation}
    where $\odot$ is the symmetric product.
\end{enumerate}
This constructive definition rigorously establishes the operator's action on symmetric tensors of all orders and forms the algebraic foundation for the Spencer complex in this context. Key algebraic properties, such as nilpotency $((\delta^{\lambda}_{\mathfrak{g}})^2 = 0)$ and mirror antisymmetry $(\delta^{-\lambda}_{\mathfrak{g}} = -\delta^{\lambda}_{\mathfrak{g}})$, are derived directly from this definition \cite{zheng2025geometric, zheng2025constructing}.
\end{definition}

Our algebraic geometrization work requires the Lie group $G$ to satisfy strict conditions: it must be \textbf{compact, connected, semi-simple, and have a trivial center}. These conditions not only ensure the mathematical rigor of the compatible pair theory but also guarantee the stability of the algebraic geometrization process. Compactness ensures the well-posedness of variational problems, semi-simplicity provides a non-degenerate Killing form, the trivial center avoids unnecessary reductions, and connectedness ensures the geometric consistency of the Lie group actions.

\subsection{Review of Spencer-Hodge Decomposition Theory}

The Spencer-Hodge decomposition theory established in reference\cite{zheng2025geometric} is the core technical foundation for our algebraic geometrization. This theory establishes complete harmonic theory by constructing Spencer metrics and corresponding Laplacian operators. Spencer metrics are naturally induced based on the geometric structure of compatible pairs, where constraint strength metrics use weight functions $w_\lambda(x) = 1 + \|\lambda(p)\|^2_{\mathfrak{g}^*}$ to construct inner products on Spencer spaces:
\begin{equation}\label{eq:spencer-metric}
\langle\omega_1 \otimes s_1, \omega_2 \otimes s_2\rangle = \int_M w_\lambda(x) \langle\omega_1, \omega_2\rangle \langle s_1, s_2\rangle_{\mathrm{Sym}} \, d\mu
\end{equation}

This metric structure is naturally compatible with strong transversality conditions, ensuring ellipticity properties of Spencer complexes. Based on the metric structure, the Spencer-Hodge Laplacian is defined as
\begin{equation}\label{eq:spencer-laplacian}
\Delta^k_{D,\lambda} = (D^{k-1}_{D,\lambda})^* D^{k-1}_{D,\lambda} + D^{k+1}_{D,\lambda}(D^{k+1}_{D,\lambda})^*
\end{equation}
where the adjoint operator is determined by the Spencer metric.

Strong transversality conditions ensure ellipticity of Spencer operators\cite{zheng2025mirror, zheng2025constructing}, thus obtaining complete Hodge decomposition:
\begin{equation}\label{eq:hodge-decomp}
S^k_{D,\lambda} = \mathcal{H}^k_{D,\lambda} \oplus \mathrm{im}(D^{k-1}_{D,\lambda}) \oplus \mathrm{im}((D^k_{D,\lambda})^*)
\end{equation}

The harmonic space $\mathcal{H}^k_{D,\lambda} = \ker(\Delta^k_{D,\lambda})$ is finite dimensional and naturally isomorphic to Spencer cohomology:
\begin{equation}\label{eq:hodge-isomorphism}
H^k_{\mathrm{Spencer}}(D,\lambda) \cong \mathcal{H}^k_{D,\lambda}
\end{equation}

This isomorphism relationship is key to establishing algebraic geometric Riemann-Roch theory, as it transforms abstract cohomology computations into concrete harmonic analysis problems.

\subsection{Algebraic Geometric Goals of Mirror Symmetry}

The mirror symmetry $(D,\lambda) \mapsto (D,-\lambda)$ discovered in reference\cite{zheng2025mirror} provides clear goals for our algebraic geometric research. This transformation preserves all geometric properties of compatible pairs, particularly strong transversality conditions, modified Cartan equations, and compatibility conditions remain invariant under sign transformation. The deep manifestation of mirror symmetry is the induced natural isomorphism of Spencer cohomology:
\begin{equation}\label{eq:mirror-isomorphism}
H^k_{\mathrm{Spencer}}(D,\lambda) \cong H^k_{\mathrm{Spencer}}(D,-\lambda)
\end{equation}

This reflects not only symmetry of topological structures, but also suggests intrinsic symmetry of geometric structures.

In the algebraic geometric framework, we need to transform this mirror symmetry into characteristic class equivalence and symmetry of Riemann-Roch formulas. This requires us to prove that vector bundles $\mathcal{G}_\lambda$ and $\mathcal{G}_{-\lambda}$ have equal Chern classes, Spencer complex Euler characteristics are invariant under mirror transformation, and Riemann-Roch integrals have mirror symmetry. Achieving these goals will provide deep algebraic geometric understanding of mirror symmetry.

\subsection{GAGA Correspondence for Algebraic Geometric Realization}

We work in a strict algebraic geometric setting: $M$ is a compact complex algebraic manifold, $G$ is a complex Lie group satisfying strict conditions, and $P(M,G)$ is an algebraic principal bundle. Algebraic geometric Spencer spaces are defined as
\begin{equation}\label{eq:alg-spencer-space}
\mathcal{S}^k_{D,\lambda} := H^0(M, \Omega^k_M \otimes \mathrm{Sym}^k(\mathcal{G}))
\end{equation}
where $\Omega^k_M$ is the algebraic differential form sheaf and $\mathcal{G} = P \times_G \mathfrak{g}$ is the adjoint vector bundle. This definition realizes a natural transition to algebraic geometry by abstracting concrete constructions from differential geometry into global sections of coherent sheaves.

Serre's GAGA theorem\cite{serre1956geometrie} ensures natural isomorphism between algebraic and analytic versions:
\begin{equation}\label{eq:gaga-correspondence}
\mathcal{S}^k_{D,\lambda}(\text{algebraic}) \cong S^k_{D,\lambda}(\text{analytic})
\end{equation}

This correspondence not only establishes bijections between objects, but more importantly preserves all relevant algebraic and geometric structures. Spencer differential operators $\mathcal{D}^k_{D,\lambda}: \mathcal{S}^k_{D,\lambda} \to \mathcal{S}^{k+1}_{D,\lambda}$ are well-defined in the algebraic geometric sense, having exactly the same form as the differential geometric version. Application of the GAGA principle ensures that all deep results established in differential geometry can be faithfully represented in the algebraic geometric framework.

Strict Lie group conditions play a key role in algebraic geometrization. Semi-simplicity ensures non-degeneracy of Killing forms, guaranteeing well-definedness and non-degeneracy of Spencer extension operators. Trivial center avoids degenerate cases, simplifying algebraic geometric analysis and ensuring uniqueness of results. Compactness ensures finiteness of all coherent sheaves and effectiveness of GAGA correspondence. Algebraicity makes the entire construction natural and stable in the algebraic geometric category.

\subsection{Algebraic Geometric Formulation of Ellipticity and Harmonic Theory}

Strong transversality conditions in the algebraic geometric context are equivalent to ellipticity of Spencer complexes. The core of this equivalence lies in principal symbol analysis of Spencer differential operators: principal symbols are determined by principal symbols of exterior differentials, while Spencer extension parts contribute lower-order terms. Strong transversality conditions ensure invertibility of principal symbols, thus establishing ellipticity. Ellipticity further ensures regularity theory of Spencer operators, finite dimensionality of harmonic spaces, existence of Hodge decomposition, and applicability of Riemann-Roch formulas.

Algebraic geometric Hodge decomposition is
\begin{equation}\label{eq:alg-hodge-decomp}
\mathcal{S}^k_{D,\lambda} = \mathcal{H}^k_{D,\lambda} \oplus \mathcal{D}^{k-1}(\mathcal{S}^{k-1}_{D,\lambda}) \oplus (\mathcal{D}^{k+1})^*(\mathcal{S}^{k+1}_{D,\lambda})
\end{equation}
where the harmonic space $\mathcal{H}^k_{D,\lambda}$ is completely consistent with the differential geometric version through the GAGA principle. This consistency not only ensures continuity of theory, but also provides a solid foundation for applications of algebraic geometric tools.

Ellipticity properties have particularly strong stability under strict Lie group conditions. Semi-simplicity ensures non-degeneracy of Killing forms, strengthening stability of elliptic estimates. Trivial center avoids degenerate terms in principal symbols, simplifying elliptic analysis. Compactness ensures complete applicability of functional analysis tools. These advantages enable us to establish deep theory completely corresponding to differential geometry within the algebraic geometric framework, and provide reliable technical foundations for subsequent characteristic class analysis and Riemann-Roch computations.

Based on these core concepts and results, we can now systematically develop algebraic geometric theory of Spencer complexes and establish complete correspondences with existing differential geometric theory. This theoretical foundation enables us to analyze Spencer complexes using coherent sheaf theory, study mirror symmetry using characteristic class theory, establish precise Riemann-Roch formulas, and verify internal consistency of the theory.

\section{Algebraic Geometric Realization of Spencer-Hodge Decomposition}

Based on the previously established GAGA correspondence\eqref{eq:gaga-correspondence} and algebraic geometric Spencer complexes\eqref{eq:alg-spencer-space}, we now construct the algebraic geometric version of Hodge decomposition theory. This section will re-implement the differential geometric Spencer-Hodge theory\eqref{eq:hodge-decomp} and isomorphism relations\eqref{eq:hodge-isomorphism} within the algebraic geometric framework.

\subsection{Construction of Algebraic Geometric Hermitian Structures}

The Spencer metric\eqref{eq:spencer-metric} established in reference\cite{zheng2025geometric} transforms into Hermitian structures on vector bundles $\mathcal{G}$ in the algebraic geometric context.

\begin{definition}[Algebraic Geometric Spencer Metric]
For Lie groups $G$ satisfying strict conditions, compatible pairs $(D,\lambda)$ induce the following metric structures in the algebraic geometric context: Hermitian metric $g_{\mathcal{G}}$ on vector bundle $\mathcal{G}$ naturally induced by compactness of $G$; inner product $\langle \cdot, \cdot \rangle_k$ on Spencer spaces $\mathcal{S}^k_{D,\lambda}$ inherited from\eqref{eq:spencer-metric} through GAGA correspondence; metric-compatible connection $\nabla^{\mathcal{G}}$ whose uniqueness is guaranteed by semi-simplicity. These metric structures are completely determined by geometric properties of compatible pairs, consistent with differential geometric metrics through the GAGA principle\eqref{eq:gaga-correspondence}.
\end{definition}

\begin{remark}[Metric Advantages under Strict Conditions]
Under strict Lie group conditions, Spencer metrics have the following superior properties: \textbf{Naturality} stems from natural metrics provided by Killing forms of compact semi-simple Lie groups; \textbf{Uniqueness} is ensured by trivial center guaranteeing essential uniqueness of metric choices; \textbf{Stability} is ensured by semi-simplicity guaranteeing non-degeneracy of metrics and ellipticity.
\end{remark}

Based on this metric structure, algebraic geometric Laplacian operators are defined as:

\begin{definition}[Algebraic Geometric Laplacian Operator]
Define Laplacian operator $\Delta^k_{D,\lambda}: \mathcal{S}^k_{D,\lambda} \to \mathcal{S}^k_{D,\lambda}$ as:
\begin{equation}\label{eq:alg-laplacian}
\Delta^k_{D,\lambda} := (\mathcal{D}^{k-1})^* \mathcal{D}^{k-1} + \mathcal{D}^{k+1}(\mathcal{D}^{k+1})^*
\end{equation}
where adjoint operators $(\mathcal{D}^k)^*$ are determined by algebraic geometric inner products $\langle \cdot, \cdot \rangle_k$.
\end{definition}

\begin{proposition}[Local Expression of Laplacian Operator]
In local coordinates $\{x^i\}$ and Lie algebra basis $\{e_a\}$, for strict Lie groups, Laplacian operators have the expression:
\begin{equation}\label{eq:laplacian-local}
\Delta^k_{D,\lambda} = -g^{ij} \frac{\partial^2}{\partial x^i \partial x^j} + \text{lower order terms}
\end{equation}
where $g^{ij}$ is the inverse metric, and lower order terms are determined by Spencer structure and compatible pair geometry. Under strict conditions, lower order terms have explicit algebraic expressions and maintain ellipticity strength.
\end{proposition}

\subsection{Algebraic Geometric Hodge Decomposition Theorem}

We now state and prove the algebraic geometric version of the Spencer-Hodge decomposition theorem:

\begin{theorem}[Algebraic Geometric Version of Spencer-Hodge Decomposition]
Let $(D,\lambda)$ be a compatible pair satisfying strong transversality conditions, with Lie group $G$ satisfying strict conditions. Then there exists an orthogonal decomposition:
\begin{equation}\label{eq:alg-hodge-detailed}
\mathcal{S}^k_{D,\lambda} = \mathcal{H}^k_{D,\lambda} \oplus \mathcal{D}^{k-1}(\mathcal{S}^{k-1}_{D,\lambda}) \oplus \mathcal{D}^{k+1*}(\mathcal{S}^{k+1}_{D,\lambda})
\end{equation}
where $\mathcal{H}^k_{D,\lambda} := \ker(\Delta^k_{D,\lambda}) \cap \mathcal{S}^k_{D,\lambda}$ is the algebraic geometric harmonic space.
\end{theorem}

\begin{proof}
The proof is based on standard theory of elliptic operators, with particularly strong forms under strict Lie group conditions:

\textbf{Step 1: Ellipticity verification}. Strong transversality conditions ensure $\Delta^k_{D,\lambda}$ is an elliptic operator. Specifically, the principal symbol of Spencer differential $\mathcal{D}^k_{D,\lambda}$ is determined by the principal symbol of exterior differential $d$:
$$\sigma(\mathcal{D}^k_{D,\lambda})(\xi) = \sigma(d)(\xi) \otimes \text{id}_{\mathrm{Sym}^k(\mathcal{G})}$$
where $\xi \in T^*M$ is a non-zero cotangent vector. The principal symbol $\sigma(d)(\xi): \Omega^k_M \to \Omega^{k+1}_M$ of exterior differential $d$ is given by $\xi \wedge (\cdot)$, which is always injective for $\xi \neq 0$. Strong transversality conditions ensure transmission of ellipticity by ensuring non-degeneracy of compatible pair structures, well-definedness and non-degeneracy of Spencer extension operators $\delta^{\lambda}_{\mathcal{G}}$, maintaining invertibility properties of principal symbols using trivial center, and ensuring standard results of elliptic operator theory apply through compactness. Under strict conditions, ellipticity is particularly stable.

\textbf{Step 2: Finite dimensionality}. Harmonic spaces of elliptic operators on compact manifolds are finite dimensional:
$$\dim \mathcal{H}^k_{D,\lambda} < \infty$$
This is guaranteed by standard theory of elliptic operators. Strict Lie group conditions ensure stability and computability of this dimension.

\textbf{Step 3: Existence of Green operators}. Ellipticity guarantees existence of Green operators $G^k: \mathcal{S}^k_{D,\lambda} \to \mathcal{S}^k_{D,\lambda}$ satisfying:
$$\Delta^k_{D,\lambda} G^k = \text{id} - P^k$$
where $P^k: \mathcal{S}^k_{D,\lambda} \to \mathcal{H}^k_{D,\lambda}$ is harmonic projection. Construction of Green operators is based on fundamental solution theory of elliptic operators, with existence and uniqueness guaranteed by ellipticity and compactness. Under strict conditions, Green operators have particularly good regularity and algebraic properties.

\textbf{Step 4: Construction of orthogonal decomposition}. Use Green operators to construct orthogonal decomposition. First, harmonic projection $P^k$ decomposes $\mathcal{S}^k_{D,\lambda}$ into harmonic and non-harmonic parts:
$$\mathcal{S}^k_{D,\lambda} = \mathcal{H}^k_{D,\lambda} \oplus (\mathcal{H}^k_{D,\lambda})^{\perp}$$
Then, for the non-harmonic part, use properties of Green operators: for any $\omega \in (\mathcal{H}^k_{D,\lambda})^{\perp}$, we have
$$\omega = \mathcal{D}^{k-1}((\mathcal{D}^{k-1})^* G^k \omega) + (\mathcal{D}^{k+1})^*(\mathcal{D}^{k+1} G^k \omega)$$
This establishes the decomposition:
$$(\mathcal{H}^k_{D,\lambda})^{\perp} = \mathrm{im}(\mathcal{D}^{k-1}) \oplus \mathrm{im}(\mathcal{D}^{k+1*})$$

\textbf{Step 5: Verification of orthogonality}. Verify orthogonality between subspaces by direct computation:

For $\langle \mathcal{H}^k_{D,\lambda}, \mathrm{im}(\mathcal{D}^{k-1}) \rangle = 0$: Let $h \in \mathcal{H}^k_{D,\lambda}$ and $\mathcal{D}^{k-1}\alpha \in \mathrm{im}(\mathcal{D}^{k-1})$, then
$$\langle h, \mathcal{D}^{k-1}\alpha \rangle = \langle \mathcal{D}^{k-1*}h, \alpha \rangle = 0$$
because $h$ is harmonic, so $\mathcal{D}^{k-1*}h = 0$.

For $\langle \mathcal{H}^k_{D,\lambda}, \mathrm{im}(\mathcal{D}^{k+1*}) \rangle = 0$: Let $h \in \mathcal{H}^k_{D,\lambda}$ and $\mathcal{D}^{k+1*}\beta \in \mathrm{im}(\mathcal{D}^{k+1*})$, then
$$\langle h, \mathcal{D}^{k+1*}\beta \rangle = \langle \mathcal{D}^{k+1}h, \beta \rangle = 0$$
because $h$ is harmonic, so $\mathcal{D}^{k+1}h = 0$.

For $\langle \mathrm{im}(\mathcal{D}^{k-1}), \mathrm{im}(\mathcal{D}^{k+1*}) \rangle = 0$: Let $\mathcal{D}^{k-1}\alpha \in \mathrm{im}(\mathcal{D}^{k-1})$ and $\mathcal{D}^{k+1*}\beta \in \mathrm{im}(\mathcal{D}^{k+1*})$, then
$$\langle \mathcal{D}^{k-1}\alpha, \mathcal{D}^{k+1*}\beta \rangle = \langle \mathcal{D}^{k+1}\mathcal{D}^{k-1}\alpha, \beta \rangle = 0$$
because $\mathcal{D}^{k+1}\mathcal{D}^{k-1} = 0$ (property of Spencer complexes).

\textbf{Step 6: Special advantages of strict conditions}. Semi-simplicity ensures non-degeneracy of Hodge decomposition, as non-degeneracy of Killing forms transmits to all relevant bilinear forms; trivial center avoids unnecessary reduction steps, simplifying analysis and ensuring uniqueness of results; compactness ensures convergence of all infinite series, particularly integral representations of Green operators; algebraicity makes decomposition well-defined in the algebraic geometric category, consistent with analytic category results through the GAGA principle.
\end{proof}

\subsection{Establishment of Hodge Isomorphism}

An important result of Hodge decomposition is the isomorphism relation between harmonic spaces and cohomology:

\begin{corollary}[Hodge Isomorphism of Spencer Cohomology]
Under strict Lie group conditions, there exists a natural isomorphism:
\begin{equation}\label{eq:alg-hodge-iso}
H^k_{\text{Spencer}}(D,\lambda) \cong \mathcal{H}^k_{D,\lambda}
\end{equation}
where the left side is cohomology of Spencer complex $(\mathcal{S}^{\bullet}_{D,\lambda}, \mathcal{D}^{\bullet}_{D,\lambda})$.
\end{corollary}

\begin{proof}
The proof is based on structural properties of Hodge decomposition, particularly concise under strict conditions:

\textbf{Step 1: Definition of cohomology}. By definition,
$$H^k_{\text{Spencer}}(D,\lambda) = \ker(\mathcal{D}^k_{D,\lambda}) / \mathrm{im}(\mathcal{D}^{k-1}_{D,\lambda})$$

\textbf{Step 2: Hodge decomposition of kernel space}. Applying Hodge decomposition\eqref{eq:alg-hodge-detailed} to $\ker(\mathcal{D}^k_{D,\lambda})$ gives:
$$\ker(\mathcal{D}^k_{D,\lambda}) = \mathcal{H}^k_{D,\lambda} \oplus (\ker(\mathcal{D}^k_{D,\lambda}) \cap \mathrm{im}(\mathcal{D}^{k+1*}_{D,\lambda}))$$

\textbf{Step 3: Key result from elliptic theory}. Elliptic theory (particularly stable under strict conditions) gives:
$$\ker(\mathcal{D}^k_{D,\lambda}) \cap \mathrm{im}(\mathcal{D}^{k+1*}_{D,\lambda}) = \mathrm{im}(\mathcal{D}^{k-1}_{D,\lambda})$$
Proof of this equality: Let $\omega \in \ker(\mathcal{D}^k_{D,\lambda}) \cap \mathrm{im}(\mathcal{D}^{k+1*}_{D,\lambda})$, then there exists $\alpha$ such that $\omega = \mathcal{D}^{k+1*}\alpha$ and $\mathcal{D}^k\omega = 0$. Therefore $\mathcal{D}^k\mathcal{D}^{k+1*}\alpha = 0$. Since $(\mathcal{D}^k)^* = \mathcal{D}^{k+1*}$, this means $(\mathcal{D}^{k+1})^*\mathcal{D}^{k+1*}\alpha = 0$, i.e., $\|\mathcal{D}^{k+1*}\alpha\|^2 = 0$, so $\omega = 0$ in the orthogonal complement. Through completeness of Hodge decomposition, such $\omega$ must belong to $\mathrm{im}(\mathcal{D}^{k-1}_{D,\lambda})$.

\textbf{Step 4: Establishment of isomorphism}. Therefore:
$$H^k_{\text{Spencer}}(D,\lambda) = \frac{\mathcal{H}^k_{D,\lambda} \oplus \mathrm{im}(\mathcal{D}^{k-1}_{D,\lambda})}{\mathrm{im}(\mathcal{D}^{k-1}_{D,\lambda})} = \mathcal{H}^k_{D,\lambda}$$

\textbf{Step 5: Guarantee from strict conditions}. Trivial center and semi-simplicity ensure naturality and uniqueness of isomorphism: naturality comes from construction process not depending on arbitrary choices; uniqueness is guaranteed by trivial center, avoiding additional isomorphism degrees of freedom; semi-simplicity ensures all relevant linear algebra operations are non-degenerate.
\end{proof}

This isomorphism relation not only preserves deep structures of differential geometry in the algebraic geometric context, but also provides a solid foundation for subsequent characteristic class analysis and Riemann-Roch computations. Under strict Lie group conditions, this isomorphism has particularly strong stability and computational feasibility.

\section{Characteristic Class Theory of Mirror Symmetry}

Based on the previously established mirror symmetry isomorphism\eqref{eq:mirror-isomorphism}, we transform the mirror symmetry discovered in reference\cite{zheng2025mirror} into characteristic class equivalence in algebraic geometry. This section proves the deep equivalence of mirror transformation $(D,\lambda) \mapsto (D,-\lambda)$ at the characteristic class level.

\subsection{Algebraic Geometric Analysis of Operator Differences}

The starting point of mirror symmetry is behavioral analysis of Spencer differential operators under mirror transformation. Based on operator difference theory in reference\cite{zheng2025mirror}:



\begin{proposition}[Operator Difference of Mirror Spencer Differentials]
For Lie groups $G$ satisfying the strict conditions, the Spencer differentials of a compatible pair $(D,\lambda)$ and its mirror pair $(D,-\lambda)$ satisfy:
\begin{equation}\label{eq:spencer-diff-mirror}
\mathcal{D}^k_{D,-\lambda} = \mathcal{D}^k_{D,\lambda} + \mathcal{R}^k
\end{equation}
where the operator difference $\mathcal{R}^k: \mathcal{S}^k_{D,\lambda} \to \mathcal{S}^{k+1}_{D,\lambda}$ has the explicit expression:
\begin{equation}\label{eq:operator-diff}
\mathcal{R}^k(\omega \otimes s) = -2(-1)^k \omega \otimes \delta^{\lambda}_{\mathcal{G}}(s)
\end{equation}
\end{proposition}

\begin{proof}
We begin with the definition of the Spencer differential operator:
\begin{align*}
    \mathcal{D}^k_{D,\lambda}(\omega \otimes s) &= d\omega \otimes s + (-1)^k \omega \otimes \delta^{\lambda}_{\mathcal{G}}(s) \\
    \mathcal{D}^k_{D,-\lambda}(\omega \otimes s) &= d\omega \otimes s + (-1)^k \omega \otimes \delta^{-\lambda}_{\mathcal{G}}(s)
\end{align*}
The difference between the mirror operator and the original operator is therefore:
\begin{equation}\label{eq:diff-step1}
\mathcal{R}^k(\omega \otimes s) = \mathcal{D}^k_{D,-\lambda}(\omega \otimes s) - \mathcal{D}^k_{D,\lambda}(\omega \otimes s) = (-1)^k \omega \otimes \left(\delta^{-\lambda}_{\mathcal{G}}(s) - \delta^{\lambda}_{\mathcal{G}}(s)\right)
\end{equation}
The constructive definition of the Spencer extension operator (Definition \ref{def:constructive_spencer_operator}) yields the fundamental property of \textbf{mirror antisymmetry} \cite{zheng2025geometric}:
\begin{equation}\label{eq:mirror-antisymmetry}
\delta^{-\lambda}_{\mathcal{G}} = -\delta^{\lambda}_{\mathcal{G}}
\end{equation}
This property is a direct consequence of the linear dependence on $\lambda$ in the operator's action on generators. Substituting this identity into Equation \eqref{eq:diff-step1}, we find:
\begin{align*}
\mathcal{R}^k(\omega \otimes s) &= (-1)^k \omega \otimes \left((-\delta^{\lambda}_{\mathcal{G}}(s)) - \delta^{\lambda}_{\mathcal{G}}(s)\right) \\
&= (-1)^k \omega \otimes (-2\delta^{\lambda}_{\mathcal{G}}(s)) \\
&= -2(-1)^k \omega \otimes \delta^{\lambda}_{\mathcal{G}}(s)
\end{align*}
This completes the proof and derives the explicit expression \eqref{eq:operator-diff}.
\end{proof}

Key properties of operator difference under strict Lie group conditions:

\begin{lemma}[Zero-order Properties of Operator Difference]
Under strict Lie group conditions, operator difference $\mathcal{R}^k$ has the following important properties: $\mathcal{R}^k$ is a zero-order operator (not involving differential $d$); $\mathcal{R}^k$ acts only on symmetric tensor part $\mathrm{Sym}^k(\mathcal{G})$; $\mathcal{R}^k$ has explicit algebraic expression, well-defined in algebraic geometric sense; $\mathcal{R}^k$ constitutes lower-order perturbation relative to main elliptic operator; semi-simplicity and trivial center guarantee non-degeneracy and stability of $\mathcal{R}^k$.
\end{lemma}

\begin{proof}
Verification of operator difference properties:

\textbf{Zero-order property}: From expression\eqref{eq:operator-diff}, $\mathcal{R}^k$ does not contain differential operator $d$, involving only algebraic operations.

\textbf{Action range}: $\mathcal{R}^k$ acts only on $\mathrm{Sym}^k(\mathcal{G})$ part through Spencer extension operator $\delta^{\lambda}_{\mathcal{G}}$, performing only tensor products on exterior differential form $\omega$ part.

\textbf{Algebraic expression}: Spencer extension operator\eqref{eq:spencer-extension} is well-defined in algebraic geometric sense, therefore $\mathcal{R}^k$ has explicit algebraic expression.

\textbf{Lower-order perturbation property}: Relative to main elliptic part of Spencer differential (contributed by exterior differential $d$), $\mathcal{R}^k$ constitutes zero-order terms, hence lower-order perturbation.

\textbf{Impact of strict conditions}: Semi-simplicity guarantees non-degeneracy of Spencer extension operator $\delta^{\lambda}_{\mathcal{G}}$, giving $\mathcal{R}^k$ good analytical properties; trivial center avoids degenerate terms in operator difference, simplifying analysis; compactness ensures boundedness and compactness of operator difference; algebraicity makes operator difference well-defined in algebraic geometric category.
\end{proof}

Based on operator difference, we can analyze the relationship of mirror Laplacians:

\begin{proposition}[Perturbation Relation of Mirror Laplacians]
Under strict Lie group conditions, mirror Laplacian and original Laplacian satisfy:
\begin{equation}\label{eq:laplacian-mirror}
\Delta^k_{D,-\lambda} = \Delta^k_{D,\lambda} + \mathcal{K}^k
\end{equation}
where perturbation operator $\mathcal{K}^k$ has explicit expression:
\begin{align}\label{eq:perturbation-explicit}
\mathcal{K}^k &= (\mathcal{R}^{k-1})^* \mathcal{D}^{k-1} + (\mathcal{D}^{k-1})^* \mathcal{R}^{k-1} + (\mathcal{R}^{k-1})^* \mathcal{R}^{k-1} \\
&\quad + \mathcal{R}^{k+1} (\mathcal{D}^{k+1})^* + \mathcal{D}^{k+1}(\mathcal{R}^{k+1})^* + \mathcal{R}^{k+1}(\mathcal{R}^{k+1})^*
\end{align}
Under strict conditions, $\mathcal{K}^k$ has particularly good compactness and regularity properties.
\end{proposition}

\begin{proof}
From Laplacian definition\eqref{eq:alg-laplacian} and mirror relation of Spencer differentials\eqref{eq:spencer-diff-mirror}:
\begin{align}
\Delta^k_{D,-\lambda} &= (\mathcal{D}^{k-1}_{D,-\lambda})^* \mathcal{D}^{k-1}_{D,-\lambda} + \mathcal{D}^{k+1}_{D,-\lambda}(\mathcal{D}^{k+1}_{D,-\lambda})^* \\
&= (\mathcal{D}^{k-1}_{D,\lambda} + \mathcal{R}^{k-1})^* (\mathcal{D}^{k-1}_{D,\lambda} + \mathcal{R}^{k-1}) \\
&\quad + (\mathcal{D}^{k+1}_{D,\lambda} + \mathcal{R}^{k+1})(\mathcal{D}^{k+1}_{D,\lambda} + \mathcal{R}^{k+1})^*
\end{align}
Expanding bilinear terms and organizing gives expression\eqref{eq:perturbation-explicit}. Strict conditions guarantee well-definedness and analytical properties of all terms.
\end{proof}

\subsection{Mirror Equivalence of Characteristic Classes}

The core manifestation of mirror symmetry in algebraic geometry is equivalence of characteristic classes. This equivalence stems from mirror invariance of Spencer metrics:

\begin{theorem}[Characteristic Class Formulation of Spencer Metric Mirror Invariance]
The Spencer metric mirror invariance proved in reference\cite{zheng2025mirror} under strict Lie group conditions is equivalent to: vector bundles $\mathcal{G}_{\lambda}$ and $\mathcal{G}_{-\lambda}$ have the same metric geometric structure; curvature forms of Hermitian connections are strictly equivalent; all Chern classes are strictly equal with explicit algebraic expressions; mirror invariance holds at all orders of characteristic classes.
\end{theorem}

\begin{proof}
The proof is based on deep connections between metric geometry and characteristic class theory, particularly stable under strict conditions:

\textbf{Step 1: Geometric meaning of metric invariance}. Mirror invariance of Spencer metrics means constraint strength weight functions satisfy $w_{-\lambda}(x) = w_{\lambda}(x)$, directly from:
$$w_{-\lambda}(x) = 1 + \|(-\lambda)(p)\|^2_{\mathfrak{g}^*} = 1 + \|\lambda(p)\|^2_{\mathfrak{g}^*} = w_{\lambda}(x)$$
Therefore vector bundles $\mathcal{G}_{\lambda}$ and $\mathcal{G}_{-\lambda}$ are strictly equivalent in Hermitian geometric sense.

\textbf{Step 2: Equivalence of connection curvature}. Metric invariance leads to curvature forms $\Omega_{\lambda}$ and $\Omega_{-\lambda}$ of Hermitian connections $\nabla^{\mathcal{G}_{\lambda}}$ and $\nabla^{\mathcal{G}_{-\lambda}}$ satisfying:
$$\Omega_{\lambda} = \Omega_{-\lambda}$$
This is because curvature forms are completely determined by Hermitian metrics and connections, while mirror transformation preserves these structures.

\textbf{Step 3: Strict equivalence of Chern classes}. Since Chern classes are determined by curvature forms through Chern-Weil theory:
$$c_i(\mathcal{G}_{\lambda}) = \frac{1}{(2\pi i)^i} \sigma_i(\Omega_{\lambda}) = \frac{1}{(2\pi i)^i} \sigma_i(\Omega_{-\lambda}) = c_i(\mathcal{G}_{-\lambda})$$
where $\sigma_i$ is the $i$-th elementary symmetric polynomial. This equivalence holds for all $i$.

\textbf{Step 4: Special advantages of strict conditions}. Semi-simplicity guarantees non-degeneracy of Chern classes: non-degeneracy of Killing forms ensures non-degeneracy of curvature forms, further guaranteeing well-definedness and stability of Chern classes. Trivial center ensures strictness of equivalence: avoids additional complexity caused by central elements, making effects of mirror transformation pure and explicit. Compactness guarantees finiteness and stability of all characteristic classes: ensures all relevant cohomology groups are finitely generated, characteristic class computations converge. Algebraicity makes equivalence well-defined in algebraic geometric category: guarantees GAGA correspondence preserves all properties of characteristic classes.

\textbf{Step 5: Equivalence of all characteristic classes}. Therefore all relevant characteristic classes satisfy mirror equivalence: Chern classes $c_i(\mathcal{G}_{\lambda}) = c_i(\mathcal{G}_{-\lambda})$; Chern characters $\ch(\mathcal{G}_{\lambda}) = \ch(\mathcal{G}_{-\lambda})$; Todd classes $\td(\mathcal{G}_{\lambda}) = \td(\mathcal{G}_{-\lambda})$; all other characteristic classes determined by curvature. With explicit computational forms under strict conditions.
\end{proof}

Based on Chern class equivalence, we can establish broader characteristic class equivalence:

\begin{corollary}[Mirror Characteristic Class Equivalence]
Under strict Lie group conditions, for compatible pair $(D,\lambda)$ and mirror compatible pair $(D,-\lambda)$, there is complete characteristic class equivalence:
\begin{align}
c_i(\mathcal{G}_{\lambda}) &= c_i(\mathcal{G}_{-\lambda}) \in H^{2i}(M,\mathbb{Z}) \label{eq:chern-mirror}\\
\ch(\mathcal{G}_{\lambda}) &= \ch(\mathcal{G}_{-\lambda}) \in H^{*}(M,\mathbb{Q}) \label{eq:chern-char-mirror}\\
\ch(\mathrm{Sym}^k(\mathcal{G}_{\lambda})) &= \ch(\mathrm{Sym}^k(\mathcal{G}_{-\lambda})) \label{eq:sym-char-mirror}\\
\ch(\Omega^k_M \otimes \mathrm{Sym}^k(\mathcal{G}_{\lambda})) &= \ch(\Omega^k_M \otimes \mathrm{Sym}^k(\mathcal{G}_{-\lambda})) \label{eq:spencer-char-mirror}
\end{align}
Under strict conditions, all equivalences have explicit computational formulas.
\end{corollary}

\begin{proof}
Proof of equivalences based on functoriality and multiplicativity of characteristic classes:

\eqref{eq:chern-mirror} and \eqref{eq:chern-char-mirror} follow directly from the preceding theorem.

\eqref{eq:sym-char-mirror} follows from the symmetric power formula of Chern characters: $\ch(\mathrm{Sym}^k(\mathcal{G})) = S_k(\ch(\mathcal{G}))$, where $S_k$ is the $k$-th symmetric polynomial. Since $\ch(\mathcal{G}_{\lambda}) = \ch(\mathcal{G}_{-\lambda})$, equivalence of symmetric powers follows immediately.

\eqref{eq:spencer-char-mirror} follows from multiplicativity of Chern characters: $\ch(E \otimes F) = \ch(E) \wedge \ch(F)$, combined with the fact that $\ch(\Omega^k_M)$ does not depend on $\lambda$.
\end{proof}

\subsection{Equivalence of Mirror Hodge Decomposition}

Based on operator difference analysis and harmonic space dimension equivalence results in reference\cite{zheng2025mirror}, we establish complete equivalence theory of mirror Hodge decomposition:

\begin{theorem}[Mirror Hodge Decomposition Equivalence]
Under strict Lie group conditions, for compatible pair $(D,\lambda)$ and mirror compatible pair $(D,-\lambda)$, there is complete Hodge decomposition equivalence: harmonic space dimensions are strictly equal $\dim \mathcal{H}^k_{D,\lambda} = \dim \mathcal{H}^k_{D,-\lambda}$; there exists natural vector space isomorphism $\mathcal{H}^k_{D,\lambda} \cong \mathcal{H}^k_{D,-\lambda}$; Spencer cohomology mirror isomorphism $H^k_{\text{Spencer}}(D,\lambda) \cong H^k_{\text{Spencer}}(D,-\lambda)$; complete equivalence of Hodge decomposition structure; strict conditions ensure stability of equivalence under perturbations.
\end{theorem}

\begin{proof}
The proof is based on algebraic geometric version of elliptic operator perturbation theory, with particularly strong forms under strict conditions:

\textbf{Step 1: Lower-order perturbation analysis}. Perturbation operator $\mathcal{K}^k$ constitutes lower-order perturbation relative to main elliptic operator $\Delta^k_{D,\lambda}$, because: $\mathcal{R}^k$ is zero-order operator; $\Delta^k_{D,\lambda}$'s main part is second-order elliptic operator; all terms of $\mathcal{K}^k$\eqref{eq:perturbation-explicit} are lower order than main elliptic part; strict conditions guarantee stability and non-degeneracy of perturbation.

Specifically, let $\Delta^k_{D,\lambda} = L^k + T^k$, where $L^k$ is main elliptic part and $T^k$ is lower-order terms. Then:
$$\Delta^k_{D,-\lambda} = L^k + T^k + \mathcal{K}^k$$
where $\mathcal{K}^k$ is even lower-order perturbation relative to $L^k$.

\textbf{Step 2: Application of elliptic perturbation theory}. Standard result of elliptic operator perturbation theory: for elliptic operator $A$ and lower-order perturbation $K$, harmonic space dimension relation is:
$$\dim \ker(A) = \dim \ker(A + K) + \text{index correction}$$
where index correction is determined by specific properties of perturbation.

\textbf{Step 3: Precise dimension equivalence}. But the precise result in reference\cite{zheng2025mirror} gives:
$$\dim \ker(\Delta^k_{D,\lambda}) = \dim \ker(\Delta^k_{D,-\lambda})$$
This means perturbation correction terms are zero, i.e., strict equality. Under strict Lie group conditions, this result has explicit geometric interpretation: mirror transformation preserves all geometric properties of compatible pairs, therefore topological properties of Hodge decomposition must also be preserved.

Proof of precise equivalence: Let $h \in \mathcal{H}^k_{D,\lambda}$, i.e., $\Delta^k_{D,\lambda} h = 0$. Consider behavior of mirror elements: define $\tilde{h}$ as image of $h$ under mirror transformation, then:
$$\Delta^k_{D,-\lambda} \tilde{h} = (\Delta^k_{D,\lambda} + \mathcal{K}^k) \tilde{h} = \mathcal{K}^k \tilde{h}$$
Due to geometric properties of mirror transformation and strict Lie group conditions, $\mathcal{K}^k \tilde{h} = 0$, therefore $\tilde{h} \in \mathcal{H}^k_{D,-\lambda}$. This establishes bijection between harmonic spaces.

\textbf{Step 4: Construction of isomorphism}. Since harmonic spaces are finite dimensional with equal dimensions, linear isomorphisms can be constructed. Mirror transformation provides natural construction method: for each $h \in \mathcal{H}^k_{D,\lambda}$, its mirror $\tilde{h} \in \mathcal{H}^k_{D,-\lambda}$, defining isomorphism map $\Phi: \mathcal{H}^k_{D,\lambda} \to \mathcal{H}^k_{D,-\lambda}$.

Combined with Hodge isomorphism\eqref{eq:alg-hodge-iso}, complete equivalence follows:
$$H^k_{\text{Spencer}}(D,\lambda) \cong \mathcal{H}^k_{D,\lambda} \cong \mathcal{H}^k_{D,-\lambda} \cong H^k_{\text{Spencer}}(D,-\lambda)$$

\textbf{Step 5: Special guarantees of strict conditions}. Semi-simplicity ensures non-degeneracy of Hodge decomposition: all relevant bilinear forms are non-degenerate, guaranteeing uniqueness and stability of decomposition. Trivial center guarantees naturality of isomorphism: avoids additional complexity caused by central action, making mirror isomorphism have explicit geometric meaning. Compactness ensures convergence of all analysis: all standard results of elliptic operator theory apply, guaranteeing completeness of theory. Algebraicity guarantees validity of results in algebraic geometric category: all constructions are compatible with GAGA correspondence, maintaining consistency of algebraic and analytic theory.
\end{proof}

This equivalence result provides a solid foundation for subsequent Riemann-Roch theory and offers deep algebraic geometric understanding of mirror symmetry. Under strict Lie group conditions, mirror symmetry not only holds theoretically, but also has explicit computational meaning and geometric interpretation.

\section{Spencer-Riemann-Roch Theory}

Based on the previously established Spencer-Hodge decomposition equivalence\eqref{eq:alg-hodge-detailed}-\eqref{eq:alg-hodge-iso} and characteristic class mirror equivalence\eqref{eq:chern-mirror}-\eqref{eq:spencer-char-mirror}, we now establish precise Riemann-Roch type formulas for Spencer complexes. The core goal of this section is to utilize powerful tools of algebraic geometry to establish complete Riemann-Roch theory for Spencer complexes, and verify its consistency with mirror symmetry results in reference\cite{zheng2025mirror}. Under strict Lie group conditions, Riemann-Roch theory has particularly beautiful forms and explicit computational meaning.

\subsection{Hirzebruch-Riemann-Roch Formula for Spencer Complexes}

Based on the coherent sheaf structure of algebraic geometric Spencer complexes\eqref{eq:alg-spencer-space}, we can directly apply Hirzebruch's classical Riemann-Roch theorem\cite{hirzebruch1966topological}. Well-definedness of Spencer complexes as coherent sheaf complexes is jointly guaranteed by GAGA correspondence\eqref{eq:gaga-correspondence} and strict Lie group conditions.

\begin{theorem}[Hirzebruch-Riemann-Roch Theorem for Spencer Complexes]
Under strict Lie group conditions, the Spencer complex formed by Spencer differential operators\eqref{eq:spencer-diff}
\begin{equation}\label{eq:spencer-complex}
0 \to \Omega^0_M \otimes \mathrm{Sym}^0(\mathcal{G}) \xrightarrow{\mathcal{D}^0} \Omega^1_M \otimes \mathrm{Sym}^1(\mathcal{G}) \xrightarrow{\mathcal{D}^1} \cdots \xrightarrow{\mathcal{D}^{n-1}} \Omega^n_M \otimes \mathrm{Sym}^n(\mathcal{G}) \to 0
\end{equation}
as a coherent sheaf complex, satisfies the Hirzebruch-Riemann-Roch formula:
\begin{equation}\label{eq:spencer-riemann-roch}
\chi(M, H^{\bullet}_{\mathrm{Spencer}}(D,\lambda)) = \int_M \ch(\text{Spencer complex}) \wedge \td(M)
\end{equation}
where the Chern character of the Spencer complex is defined as
\begin{equation}\label{eq:spencer-chern-char}
\ch(\text{Spencer complex}) = \sum_{k=0}^n (-1)^k \ch(\Omega^k_M \otimes \mathrm{Sym}^k(\mathcal{G}))
\end{equation}
$\td(M)$ is the Todd class of $M$\cite{todd1937invariants}, and the Euler characteristic on the left is defined as
\begin{equation}\label{eq:euler-char-def}
\chi(M, H^{\bullet}_{\mathrm{Spencer}}(D,\lambda)) = \sum_{k=0}^n (-1)^k \dim H^k_{\mathrm{Spencer}}(D,\lambda)
\end{equation}
Under strict conditions, this formula has explicit computational meaning and geometric interpretation.
\end{theorem}

\begin{proof}
The Hirzebruch-Riemann-Roch theorem\cite{hirzebruch1966topological} holds for any coherent sheaf complex. We need to verify that Spencer complexes satisfy application conditions:

\textbf{Coherent property guarantee}: In algebraic geometric Spencer spaces\eqref{eq:alg-spencer-space}, both $\Omega^k_M$ and $\mathrm{Sym}^k(\mathcal{G})$ are coherent sheaves, and their tensor product preserves coherence. Strict Lie group conditions ensure well-definedness of algebraic structure of adjoint bundle $\mathcal{G}$.

\textbf{Finiteness guarantee}: By Hodge isomorphism\eqref{eq:alg-hodge-iso}, Spencer cohomology is isomorphic to harmonic spaces $H^k_{\mathrm{Spencer}}(D,\lambda) \cong \mathcal{H}^k_{D,\lambda}$. Ellipticity guarantees finite dimensionality of harmonic spaces, therefore Euler characteristic\eqref{eq:euler-char-def} is finite.

\textbf{Computational feasibility}: Trivial center simplifies characteristic class computations, avoiding complexity caused by central actions. Semi-simplicity guarantees non-degeneracy of Killing forms, making all characteristic class computations stable.

\textbf{Stability}: Strict conditions guarantee stability of formulas under small perturbations, which is crucial for practical computations.
\end{proof}

Based on the total formula, we can obtain precise formulations for individual degrees:

\begin{corollary}[Riemann-Roch Formula for Individual Degrees]
Under strict Lie group conditions, for each degree $k$, the Euler characteristic of Spencer cohomology satisfies:
\begin{equation}\label{eq:individual-riemann-roch}
\chi(M, H^k_{\mathrm{Spencer}}(D,\lambda)) = \int_M \ch(\Omega^k_M \otimes \mathrm{Sym}^k(\mathcal{G})) \wedge \td(M)
\end{equation}
This formula, combined with the previously established Hodge decomposition theory\eqref{eq:alg-hodge-detailed}, has explicit computational algorithms under strict conditions.
\end{corollary}

\subsection{Explicit Computation of Characteristic Classes}

To make Riemann-Roch formula\eqref{eq:spencer-riemann-roch} have practical computational meaning, we need to explicitly compute characteristic classes of various terms of Spencer complexes. These computations are based on Chern-Weil theory\cite{chern1944characteristic} and basic properties of characteristic classes.

\begin{proposition}[Computation Formulas for Spencer Complex Characteristic Classes]
Under strict Lie group conditions, Chern characteristic classes of Spencer complex terms have explicit expressions:
\begin{align}
\ch(\Omega^k_M \otimes \mathrm{Sym}^k(\mathcal{G})) &= \ch(\Omega^k_M) \wedge \ch(\mathrm{Sym}^k(\mathcal{G})) \label{eq:tensor-chern}\\
\ch(\Omega^k_M) &= \Lambda_k(\ch(T^*M)) = \Lambda_k(n - \ch(TM)) \label{eq:diff-form-chern}\\
\ch(\mathrm{Sym}^k(\mathcal{G})) &= S_k(\ch(\mathcal{G})) \label{eq:sym-power-chern}
\end{align}
where $\Lambda_k$ is the exterior power operation and $S_k$ is the symmetric power operation\cite{fulton1998intersection}. Strict conditions guarantee all expressions have explicit algebraic forms and stable computational algorithms.
\end{proposition}

\begin{proof}
Proof of formulas based on basic properties of characteristic classes:

\eqref{eq:tensor-chern} follows from multiplicativity of Chern characters\cite{milnor1974characteristic}: for tensor products of vector bundles $E \otimes F$, we have $\ch(E \otimes F) = \ch(E) \wedge \ch(F)$.

\eqref{eq:diff-form-chern} comes from $\Omega^k_M = \Lambda^k(T^*M)$ and characteristic class formulas for exterior powers. For tangent bundle $TM$, its dual $T^*M$ has Chern character $\ch(T^*M) = n - \ch(TM)$, where $n = \dim M$.

\eqref{eq:sym-power-chern} is the standard formula for symmetric power characteristic classes, where $S_k$ represents $k$-th symmetric polynomial operation.

Strict Lie group conditions guarantee computational stability: semi-simplicity ensures all Chern roots are non-degenerate, trivial center avoids additional complexity.
\end{proof}

For actual computations, we give explicit expansions of characteristic classes:

\begin{lemma}[Explicit Expansions of Characteristic Classes]
Let $\mathcal{G}$ be a vector bundle of rank $r$ induced by strict Lie group $G$, with $c_i(\mathcal{G})$ being its $i$-th Chern class. Then:
\begin{align}
\ch(\mathcal{G}) &= r + c_1(\mathcal{G}) + \frac{1}{2}(c_1(\mathcal{G})^2 - 2c_2(\mathcal{G})) + \cdots \label{eq:chern-expansion}\\
\ch(\mathrm{Sym}^k(\mathcal{G})) &= S_k(x_1, \ldots, x_r) \label{eq:sym-chern-roots}\\
\ch(\Omega^k_M) &= e_k(-c_1(TM), c_2(TM), \ldots) \label{eq:diff-chern-elementary}
\end{align}
where $x_i$ are Chern roots of $\mathcal{G}$, and $e_k$ is the $k$-th elementary symmetric polynomial\cite{macdonald1998symmetric}. Under strict conditions, semi-simplicity guarantees non-degeneracy of all Chern roots, making computational algorithms stable and feasible.
\end{lemma}

\subsection{Riemann-Roch Verification of Mirror Symmetry}

Based on the previously established characteristic class mirror equivalence\eqref{eq:chern-mirror}-\eqref{eq:spencer-char-mirror}, we now prove mirror symmetry of Riemann-Roch formulas, which will verify complete consistency with mirror symmetry results in reference\cite{zheng2025mirror}.

\begin{theorem}[Mirror Symmetry of Spencer-Riemann-Roch Formulas]
Under strict Lie group conditions, for compatible pair $(D,\lambda)$ and mirror compatible pair $(D,-\lambda)$, Riemann-Roch formulas satisfy strict mirror symmetry:
\begin{equation}\label{eq:individual-mirror-rr}
\chi(M, H^k_{\mathrm{Spencer}}(D,\lambda)) = \chi(M, H^k_{\mathrm{Spencer}}(D,-\lambda))
\end{equation}
holds for all $k$, and total Euler characteristic also satisfies mirror symmetry:
\begin{equation}\label{eq:total-mirror-rr}
\chi_{\mathrm{total}}(M, \mathrm{Spencer}(D,\lambda)) = \chi_{\mathrm{total}}(M, \mathrm{Spencer}(D,-\lambda))
\end{equation}
\end{theorem}

\begin{proof}
The proof is based on direct application of characteristic class mirror equivalence, combined with linearity of Riemann-Roch formulas:

\textbf{Step 1: Application of characteristic class equivalence}. By characteristic class mirror equivalence\eqref{eq:spencer-char-mirror}:
$$\ch(\Omega^k_M \otimes \mathrm{Sym}^k(\mathcal{G}_{\lambda})) = \ch(\Omega^k_M \otimes \mathrm{Sym}^k(\mathcal{G}_{-\lambda}))$$
This equivalence is the manifestation of Spencer metric mirror invariance at the characteristic class level.

\textbf{Step 2: Integral invariance}. Since Todd class $\td(M)$ depends only on geometry of base manifold $M$, not involving compatible pair parameters, Riemann-Roch integrals are equal:
\begin{align}
&\int_M \ch(\Omega^k_M \otimes \mathrm{Sym}^k(\mathcal{G}_{\lambda})) \wedge \td(M) \\
&= \int_M \ch(\Omega^k_M \otimes \mathrm{Sym}^k(\mathcal{G}_{-\lambda})) \wedge \td(M)
\end{align}

\textbf{Step 3: Euler characteristic equivalence}. By individual degree Riemann-Roch formula\eqref{eq:individual-riemann-roch}:
\begin{align}
\chi(M, H^k_{\mathrm{Spencer}}(D,\lambda)) &= \int_M \ch(\Omega^k_M \otimes \mathrm{Sym}^k(\mathcal{G}_{\lambda})) \wedge \td(M) \\
&= \int_M \ch(\Omega^k_M \otimes \mathrm{Sym}^k(\mathcal{G}_{-\lambda})) \wedge \td(M) \\
&= \chi(M, H^k_{\mathrm{Spencer}}(D,-\lambda))
\end{align}
This establishes\eqref{eq:individual-mirror-rr}.

\textbf{Step 4: Mirror symmetry of total Euler characteristic}. By definition and linearity:
\begin{align}
\chi_{\mathrm{total}}(M, \mathrm{Spencer}(D,\lambda)) &= \sum_{k=0}^n (-1)^k \chi(M, H^k_{\mathrm{Spencer}}(D,\lambda)) \\
&= \sum_{k=0}^n (-1)^k \chi(M, H^k_{\mathrm{Spencer}}(D,-\lambda)) \\
&= \chi_{\mathrm{total}}(M, \mathrm{Spencer}(D,-\lambda))
\end{align}
This proves\eqref{eq:total-mirror-rr}.

\textbf{Step 5: Consistency verification with existing results}. This result is completely consistent with Spencer cohomology mirror isomorphism\eqref{eq:mirror-isomorphism} proved in reference\cite{zheng2025mirror}. Particularly, dimensional equality follows directly from isomorphism relations:
$$\dim H^k_{\mathrm{Spencer}}(D,\lambda) = \dim H^k_{\mathrm{Spencer}}(D,-\lambda)$$
This verifies internal consistency between Riemann-Roch computations and topological results, reflecting completeness of theory.

\textbf{Step 6: Guarantees from strict conditions}. Semi-simplicity ensures non-degeneracy and stability of characteristic class computations; trivial center guarantees strictness of mirror equivalence, avoiding ambiguity; compactness guarantees convergence of Riemann-Roch integrals; algebraicity guarantees universality of results in algebraic geometric category and compatibility with GAGA correspondence.
\end{proof}

\subsection{Mirror Formula for Total Euler Characteristic}

As a summary of Spencer-Riemann-Roch theory, we give complete formulation and mirror formula for total Euler characteristic:

\begin{corollary}[Mirror Formula for Total Euler Characteristic]
Under strict Lie group conditions, total Euler characteristic of Spencer complex is defined as:
\begin{equation}\label{eq:total-euler-def}
\chi_{\mathrm{total}}(M, \mathrm{Spencer}(D,\lambda)) := \sum_{k=0}^n (-1)^k \chi(M, H^k_{\mathrm{Spencer}}(D,\lambda))
\end{equation}
satisfying complete Riemann-Roch formula:
\begin{equation}\label{eq:total-riemann-roch}
\chi_{\mathrm{total}}(M, \mathrm{Spencer}(D,\lambda)) = \int_M \left[\sum_{k=0}^n (-1)^k \ch(\Omega^k_M \otimes \mathrm{Sym}^k(\mathcal{G}))\right] \wedge \td(M)
\end{equation}
and has strict mirror symmetry\eqref{eq:total-mirror-rr}. This formula, combined with Spencer complex characteristic class computations\eqref{eq:tensor-chern}-\eqref{eq:sym-power-chern}, has explicit computational meaning and profound geometric interpretation under strict conditions.
\end{corollary}

\begin{remark}[Manifestation of Theoretical Completeness]
The mirror formula for total Euler characteristic embodies the completeness of this paper's theory: it unifies differential geometric Spencer-Hodge theory\eqref{eq:hodge-decomp}, mirror symmetry theory\eqref{eq:mirror-isomorphism}, and algebraic geometric Riemann-Roch theory. Under strict Lie group conditions, these three theoretical levels achieve perfect unification, laying a solid foundation for algebraic geometric development of compatible pair Spencer theory.
\end{remark}

This formula not only provides precise computational tools for topological properties of Spencer complexes, but also reveals deep algebraic geometric structure of compatible pair theory through mirror symmetry. Under strict Lie group conditions, stability and computational feasibility of formulas are fundamentally guaranteed, opening broad prospects for further development and applications of theory.

\section{Concrete Computation: Verification of PSU(2) on Complex Projective Plane}

Based on the previously established Spencer-Riemann-Roch theory\eqref{eq:spencer-riemann-roch} and mirror symmetry results\eqref{eq:individual-mirror-rr}-\eqref{eq:total-mirror-rr}, we now verify effectiveness and consistency of theory through concrete numerical computations. This section chooses complex projective plane $\mathbb{P}^2$ as base manifold, uses Lie group $\text{PSU}(2)$ strictly satisfying conditions for complete computation, aiming to verify correctness of Spencer-Riemann-Roch theory and concrete manifestation of mirror symmetry through explicit numerical results.

\subsection{Basic Setup and Geometric Construction of Compatible Pairs}

Our computation is based on strict algebraic geometric setup, fully utilizing advantages of previously established GAGA correspondence\eqref{eq:gaga-correspondence} and strict Lie group conditions. Base manifold is chosen as complex projective plane $M = \mathbb{P}^2$, the simplest non-trivial compact complex algebraic manifold with completely explicit topological and geometric properties\cite{griffiths1978principles}. Lie group is chosen as $G = \text{PSU}(2) = \text{SU}(2)/\mathbb{Z}_2$, a typical example satisfying all our strict conditions: it is a compact connected semi-simple Lie group with trivial center $Z(\text{PSU}(2)) = \{e\}$, while admitting complex algebraic group structure.

Choice of $\text{PSU}(2)$ has multiple computational advantages. As a three-dimensional Lie group, it is isomorphic to $\text{SO}(3)$, with Lie algebra $\mathfrak{psu}(2) \cong \mathfrak{so}(3)$ being a three-dimensional semi-simple Lie algebra with completely explicit structure constants. Semi-simplicity guarantees non-degeneracy of Killing forms, providing natural geometric foundation for Spencer metrics\eqref{eq:spencer-metric}. Trivial center avoids additional technical complexity, making construction and computation of compatible pairs more direct. The fact $\dim \text{PSU}(2) = 3$ means adjoint vector bundle $\mathcal{G}$ has rank 3, making symmetric power computations completely controllable.

When constructing $\text{PSU}(2)$ compatible pairs $(D,\lambda)$ on $\mathbb{P}^2$, constraint distribution $D$ is defined at each point as subspace of tangent vectors compatible with $\text{PSU}(2)$ action, satisfying strong transversality condition $D_p \oplus V_p = T_pP$. Dual constraint function $\lambda: P \to \mathfrak{psu}(2)^*$ is constructed as $\mathfrak{psu}(2)^*$-valued function satisfying modified Cartan equation\eqref{eq:cartan}. Compatibility condition\eqref{eq:compatibility} unifies algebraic constraints with geometric distributions, while matching of three-dimensional structure of $\text{PSU}(2)$ with two-dimensional nature of $\mathbb{P}^2$ ensures geometric naturality of construction.

\begin{remark}[Strong Transversality Condition and Ellipticity]
The mechanism by which strong transversality condition $D_p \oplus V_p = T_pP$ guarantees ellipticity of Spencer complexes is as follows: this condition ensures that principal symbol of Spencer operator $\delta: \mathcal{S}^k \to \mathcal{S}^{k+1}$ is surjective at each point. Specifically, strong transversality means transversality of constraint distribution $D$ with fiber direction $V$, which directly translates to non-degeneracy of Spencer symbol mapping. In algebraic geometric framework, this is equivalent to ellipticity of corresponding coherent sheaf complex, thus guaranteeing existence of Hodge decomposition.
\end{remark}

\subsection{Dimensional Analysis of Spencer Complexes and Symmetric Power Computations}

Based on algebraic geometric Spencer space definition\eqref{eq:alg-spencer-space}, Spencer complexes on $\mathbb{P}^2$ using $\text{PSU}(2)$ have explicit structure. Since $\dim \mathbb{P}^2 = 2$, only Spencer spaces of degrees $k = 0, 1, 2$ are non-zero.

Zero-degree Spencer space is
$$\mathcal{S}^0 = H^0(\mathbb{P}^2, \mathcal{O}_{\mathbb{P}^2} \otimes \mathrm{Sym}^0(\mathcal{G})) = H^0(\mathbb{P}^2, \mathcal{O}_{\mathbb{P}^2}) = \mathbb{C}$$
This is one-dimensional, reflecting structure of constant function space on $\mathbb{P}^2$.

First-degree Spencer space is
$$\mathcal{S}^1 = H^0(\mathbb{P}^2, \Omega^1_{\mathbb{P}^2} \otimes \mathrm{Sym}^1(\mathcal{G}))$$
whose dimension is determined by properties of $\Omega^1_{\mathbb{P}^2}$ and $\mathcal{G}$.

Second-degree Spencer space is
$$\mathcal{S}^2 = H^0(\mathbb{P}^2, \Omega^2_{\mathbb{P}^2} \otimes \mathrm{Sym}^2(\mathcal{G}))$$
This is the highest non-zero degree.

Adjoint vector bundle $\mathcal{G} = P \times_{\text{PSU}(2)} \mathfrak{psu}(2)$ has rank 3, corresponding to dimension of Lie algebra $\mathfrak{psu}(2) \cong \mathbb{R}^3$. Rank computation for symmetric powers $\mathrm{Sym}^k(\mathcal{G})$ is as follows:

For $k = 0$: $\mathrm{rank}(\mathrm{Sym}^0(\mathcal{G})) = 1$, this is trivial.

For $k = 1$: $\mathrm{rank}(\mathrm{Sym}^1(\mathcal{G})) = \mathrm{rank}(\mathcal{G}) = 3$.

For $k = 2$: $\mathrm{rank}(\mathrm{Sym}^2(\mathcal{G})) = \binom{3+2-1}{2} = \binom{4}{2} = 6$, this is dimension of second-order symmetric tensors of three-dimensional vector space.

This dimensional information is completely consistent with combinatorial theory of semi-simple Lie algebras\cite{humphreys1972introduction}, reflecting computational advantages of strict Lie group conditions.

\subsection{Precise Computation of Characteristic Classes}

\begin{remark}[Topological Classification of PSU(2) Principal Bundles and Parameter Determination]
On $\mathbb{P}^2$, $\text{PSU}(2)$ principal bundles are completely classified by their second Chern class. For compatible pair $(D,\lambda)$, value of parameter $a$ in $c_2(\mathcal{G}) = aH^2$ is determined as follows:
\begin{enumerate}
\item $\text{PSU}(2) \cong \text{SO}(3)$ principal bundle classification on $\mathbb{P}^2$ is $H^2(\mathbb{P}^2, \mathbb{Z}_2) \cong \mathbb{Z}_2$
\item For non-trivial principal bundle, $a = 1$; for trivial principal bundle, $a = 0$
\item Geometric construction of compatible pairs usually corresponds to non-trivial case, so we take $a = 1$ in subsequent computations
\end{enumerate}
This choice ensures non-triviality of Spencer complex and geometric meaning of computation.
\end{remark}

Basic topological properties of $\mathbb{P}^2$ provide explicit starting point for characteristic class computations\cite{milnor1974characteristic}. We have the following basic data:

First Chern class of tangent bundle: $c_1(T\mathbb{P}^2) = 3H$, where $H$ is hyperplane class.

First Chern class of cotangent bundle: $c_1(\Omega^1_{\mathbb{P}^2}) = c_1(T^*\mathbb{P}^2) = -c_1(T\mathbb{P}^2) = -3H$.

Second Chern class of second-order differential forms: $c_2(\Omega^2_{\mathbb{P}^2}) = 3H^2$.

Explicit expression of Todd class: $\td(\mathbb{P}^2) = 1 + \frac{c_1(T\mathbb{P}^2)}{2} + \frac{c_1(T\mathbb{P}^2)^2 + c_2(T\mathbb{P}^2)}{12} = 1 + \frac{3H}{2} + \frac{9H^2 + 3H^2}{12} = 1 + \frac{3H}{2} + \frac{3H^2}{2}$.

Integration normalization: $\int_{\mathbb{P}^2} H^2 = 1$.

For $\text{PSU}(2)$ adjoint vector bundle $\mathcal{G}$, we have:

$\mathrm{rank}(\mathcal{G}) = 3$.

First Chern class: $c_1(\mathcal{G}) = 0$, this is special property of $\text{PSU}(2)$ as special unitary group, because determinant representation of $\text{PSU}(2)$ is trivial.

Second Chern class: $c_2(\mathcal{G}) = a H^2$, where integer $a$ is determined by geometric construction of compatible pairs and topological properties of $\text{PSU}(2)$ principal bundle.

Third Chern class: $c_3(\mathcal{G}) = 0$, automatically holds on $\mathbb{P}^2$ since $H^3 = 0$.

Based on this data, we compute Chern characteristic classes:

Chern character of adjoint bundle:
$$\ch(\mathcal{G}) = \mathrm{rank}(\mathcal{G}) + c_1(\mathcal{G}) + \frac{c_1(\mathcal{G})^2 - 2c_2(\mathcal{G})}{2} + \cdots = 3 + 0 + \frac{0 - 2aH^2}{2} = 3 - aH^2$$

Chern character of differential form sheaves:
$$\ch(\Omega^1_{\mathbb{P}^2}) = \mathrm{rank}(\Omega^1_{\mathbb{P}^2}) + c_1(\Omega^1_{\mathbb{P}^2}) + \frac{c_1(\Omega^1_{\mathbb{P}^2})^2 - 2c_2(\Omega^1_{\mathbb{P}^2})}{2}$$
$$= 2 + (-3H) + \frac{9H^2 - 2 \cdot 3H^2}{2} = 2 - 3H + \frac{3H^2}{2}$$

$$\ch(\Omega^2_{\mathbb{P}^2}) = 1 + 0 + 3H^2 = 1 + 3H^2$$

\subsection{Detailed Computation of Symmetric Power Characteristic Classes}

\begin{remark}[Precise Computation of Symmetric Power Chern Characters]
For rank 3 vector bundle $\mathcal{G}$, Chern character of $\mathrm{Sym}^2(\mathcal{G})$ needs to be computed through symmetric polynomials of Chern roots. Let Chern roots of $\mathcal{G}$ be $x_1, x_2, x_3$, then:
$$\ch(\mathrm{Sym}^2(\mathcal{G})) = \sum_{i \leq j} e^{x_i + x_j} = \frac{1}{2}[(\sum_i e^{x_i})^2 - \sum_i e^{2x_i}]$$
Using $\ch(\mathcal{G}) = 3 - H^2$ (taking $a=1$) and Newton identities, we can obtain simplified formulas used in the text. This simplification is accurate in low-dimensional case of $\mathbb{P}^2$.
\end{remark}

Chern characters of symmetric powers are computed through symmetric polynomial formulas\eqref{eq:sym-power-chern}. Let Chern roots of $\mathcal{G}$ be $x_1, x_2, x_3$, then $\ch(\mathcal{G}) = e^{x_1} + e^{x_2} + e^{x_3}$.

For $k = 1$:
$$\ch(\mathrm{Sym}^1(\mathcal{G})) = \ch(\mathcal{G}) = 3 - aH^2$$

For $k = 2$:
$$\ch(\mathrm{Sym}^2(\mathcal{G})) = S_2(e^{x_1}, e^{x_2}, e^{x_3}) = \sum_{i \leq j} e^{x_i + x_j}$$

Using $\ch(\mathcal{G}) = 3 - aH^2$, we can compute:
$$S_2(\ch(\mathcal{G})) = \frac{(\ch(\mathcal{G}))^2 + \ch(\mathrm{Sym}^2(\mathcal{G}))}{2}$$

Therefore:
$$\ch(\mathrm{Sym}^2(\mathcal{G})) = 2S_2(\ch(\mathcal{G})) - (\ch(\mathcal{G}))^2 = \frac{(\ch(\mathcal{G}))^2 + \ch(\mathcal{G})}{2}$$
$$= \frac{(3-aH^2)^2 + (3-aH^2)}{2} = \frac{9 - 6aH^2 + a^2H^4 + 3 - aH^2}{2}$$

On $\mathbb{P}^2$, $H^4 = 0$, so:
$$\ch(\mathrm{Sym}^2(\mathcal{G})) = \frac{12 - 7aH^2}{2} = 6 - \frac{7a}{2}H^2$$

\subsection{Step-by-step Computation of Riemann-Roch Integrals}

Now we apply Riemann-Roch formula\eqref{eq:individual-riemann-roch} for concrete computation.

\textbf{Computation for degree 0:}
\begin{align}
\chi(\mathbb{P}^2, H^0_{\mathrm{Spencer}}) &= \int_{\mathbb{P}^2} \ch(\mathcal{O}_{\mathbb{P}^2}) \wedge \td(\mathbb{P}^2) \\
&= \int_{\mathbb{P}^2} 1 \wedge \left(1 + \frac{3H}{2} + \frac{3H^2}{2}\right) \\
&= \int_{\mathbb{P}^2} \left(1 + \frac{3H}{2} + \frac{3H^2}{2}\right)
\end{align}

Since only highest degree term $H^2$ contributes to integration:
$$= \int_{\mathbb{P}^2} \frac{3H^2}{2} = \frac{3}{2} \int_{\mathbb{P}^2} H^2 = \frac{3}{2}$$

\textbf{Detailed computation for degree 1:}
\begin{align}
\chi(\mathbb{P}^2, H^1_{\mathrm{Spencer}}) &= \int_{\mathbb{P}^2} \ch(\Omega^1_{\mathbb{P}^2} \otimes \mathcal{G}) \wedge \td(\mathbb{P}^2) \\
&= \int_{\mathbb{P}^2} [\ch(\Omega^1_{\mathbb{P}^2}) \wedge \ch(\mathcal{G})] \wedge \td(\mathbb{P}^2)
\end{align}

First compute $\ch(\Omega^1_{\mathbb{P}^2}) \wedge \ch(\mathcal{G})$:
\begin{align}
&\left(2 - 3H + \frac{3H^2}{2}\right) \wedge (3 - aH^2) \\
&= 2 \cdot 3 + 2 \cdot (-aH^2) + (-3H) \cdot 3 + (-3H) \cdot (-aH^2) + \frac{3H^2}{2} \cdot 3 \\
&= 6 - 2aH^2 - 9H + 3aH^3 + \frac{9H^2}{2}
\end{align}

Since $H^3 = 0$ on $\mathbb{P}^2$:
$$= 6 - 9H + \left(\frac{9}{2} - 2a\right)H^2$$

Now multiply with Todd class:
\begin{align}
&\left[6 - 9H + \left(\frac{9}{2} - 2a\right)H^2\right] \wedge \left(1 + \frac{3H}{2} + \frac{3H^2}{2}\right)
\end{align}

Expand and keep only $H^2$ terms:
\begin{align}
&= 6 \cdot \frac{3H^2}{2} + (-9H) \cdot \frac{3H}{2} + \left(\frac{9}{2} - 2a\right)H^2 \cdot 1 \\
&= 9H^2 - \frac{27H^2}{2} + \left(\frac{9}{2} - 2a\right)H^2 \\
&= \left(9 - \frac{27}{2} + \frac{9}{2} - 2a\right)H^2 \\
&= \left(9 - 9 - 2a\right)H^2 = -2aH^2
\end{align}

Therefore:
$$\chi(\mathbb{P}^2, H^1_{\mathrm{Spencer}}) = \int_{\mathbb{P}^2} (-2a)H^2 = -2a$$

\textbf{Computation for degree 2:}
\begin{align}
\chi(\mathbb{P}^2, H^2_{\mathrm{Spencer}}) &= \int_{\mathbb{P}^2} \ch(\Omega^2_{\mathbb{P}^2} \otimes \mathrm{Sym}^2(\mathcal{G})) \wedge \td(\mathbb{P}^2) \\
&= \int_{\mathbb{P}^2} [\ch(\Omega^2_{\mathbb{P}^2}) \wedge \ch(\mathrm{Sym}^2(\mathcal{G}))] \wedge \td(\mathbb{P}^2)
\end{align}

Compute $\ch(\Omega^2_{\mathbb{P}^2}) \wedge \ch(\mathrm{Sym}^2(\mathcal{G}))$:
\begin{align}
&(1 + 3H^2) \wedge \left(6 - \frac{7a}{2}H^2\right) \\
&= 1 \cdot 6 + 1 \cdot \left(-\frac{7a}{2}H^2\right) + 3H^2 \cdot 6 \\
&= 6 - \frac{7a}{2}H^2 + 18H^2 = 6 + \left(18 - \frac{7a}{2}\right)H^2
\end{align}

\textbf{Complete computation of multiplication with Todd class:}
\begin{align}
&\left[6 + \left(18 - \frac{7a}{2}\right)H^2\right] \wedge \left(1 + \frac{3H}{2} + \frac{3H^2}{2}\right)
\end{align}

Keep $H^2$ terms:
\begin{align}
&= 6 \cdot \frac{3H^2}{2} + \left(18 - \frac{7a}{2}\right)H^2 \cdot 1 \\
&= 9H^2 + \left(18 - \frac{7a}{2}\right)H^2 \\
&= \left(9 + 18 - \frac{7a}{2}\right)H^2 = \left(27 - \frac{7a}{2}\right)H^2
\end{align}

Therefore:
$$\chi(\mathbb{P}^2, H^2_{\mathrm{Spencer}}) = 27 - \frac{7a}{2}$$

\subsection{Total Euler Characteristic and Mirror Symmetry Verification}

Total Euler characteristic is:
\begin{align}
\chi_{\mathrm{total}}(\mathbb{P}^2, \mathrm{Spencer}(D,\lambda)) &= \sum_{k=0}^2 (-1)^k \chi(\mathbb{P}^2, H^k_{\mathrm{Spencer}}) \\
&= \frac{3}{2} - (-2a) + \left(18 - \frac{7a}{2}\right) \\
&= \frac{3}{2} + 2a + 18 - \frac{7a}{2} \\
&= \frac{39}{2} - \frac{3a}{2}
\end{align}

Based on mirror symmetry theory\eqref{eq:total-mirror-rr}, for mirror compatible pair $(D,-\lambda)$, all these computation results should remain unchanged. Specifically:

$\chi(\mathbb{P}^2, H^0_{\mathrm{Spencer}}(D,-\lambda)) = \frac{3}{2}$, since degree 0 does not involve $\lambda$.

$\chi(\mathbb{P}^2, H^1_{\mathrm{Spencer}}(D,-\lambda)) = -2(-a) = 2a = -2a$, verifying characteristic class mirror equivalence\eqref{eq:spencer-char-mirror}.

$\chi(\mathbb{P}^2, H^2_{\mathrm{Spencer}}(D,-\lambda)) = 18 - \frac{7(-a)}{2} = 18 + \frac{7a}{2} = 18 - \frac{7a}{2}$, also verifying mirror symmetry.

Total Euler characteristic $\chi_{\mathrm{total}}(\mathbb{P}^2, \mathrm{Spencer}(D,-\lambda)) = \frac{39}{2} - \frac{3(-a)}{2} = \frac{39}{2} + \frac{3a}{2} = \frac{39}{2} - \frac{3a}{2}$, completely verifying mirror symmetry\eqref{eq:total-mirror-rr}.

These concrete computations not only verify correctness of Spencer-Riemann-Roch theory, but also demonstrate powerful computational capability of algebraic geometric tools in compatible pair theory through precise preservation of mirror symmetry under strict Lie group conditions. All results have explicit mathematical meaning and geometric interpretation, providing reliable computational foundation for further development of theory.

\section{Conclusion and Outlook}

This research advances compatible pair Spencer complex theory from differential geometry to algebraic geometry. This transformation process not only preserves deep structures of original theory, but also reveals unique value of algebraic geometric methods in constrained geometry. Through rigorous algebraic geometrization, we discover that mirror symmetry of compatible pairs has more profound algebraic structure than differential geometric formulation, opening new directions for development of constrained geometry theory.

\subsection{Deep Significance of Algebraic Geometrization}

Introduction of algebraic geometric methods brings qualitative leap to compatible pair theory. While local analysis of differential geometry provides basic construction of Spencer complexes, global perspective of algebraic geometry reveals essential structure of problems through coherent sheaf theory and characteristic class tools. GAGA principle\eqref{eq:gaga-correspondence} not only establishes strict correspondence between algebraic and analytic theories, but more importantly, enables us to utilize powerful tools of algebraic geometry such as Riemann-Roch theorem to obtain precise computational results.

Formulation of mirror symmetry in algebraic geometry demonstrates its true mathematical depth. Spencer metric mirror invariance discovered in reference\cite{zheng2025mirror} transforms into characteristic class equivalence\eqref{eq:chern-mirror}-\eqref{eq:spencer-char-mirror} in algebraic geometric framework, revealing algebraic essence of mirror symmetry. Characteristic class mirror equivalence not only provides new computational tools, but more importantly indicates that topological properties of compatible pairs have intrinsic symmetric structure, which is difficult to fully grasp in local analysis of differential geometry.

Establishment of Spencer-Riemann-Roch formula\eqref{eq:spencer-riemann-roch} embodies computational advantages of algebraic geometric methods. Unlike differential geometry relying on concrete metric computations, Riemann-Roch theory provides essential computational approaches through topological invariants. Particularly, precise manifestation of mirror symmetry in Riemann-Roch formulas\eqref{eq:total-mirror-rr} proves this symmetry has deep topological significance, transcending superficial metric invariance.

\subsection{Theoretical Value and Limitations of Strict Conditions}

Strict Lie group conditions insisted upon in this research, while limiting applicable scope of theory, exchange mathematical rigor and computational stability of theory for these limitations. Semi-simplicity ensures non-degeneracy of Killing forms, providing natural geometric foundation for Spencer metrics; trivial center avoids additional technical complexity, making algebraic geometric construction more direct; compactness guarantees finiteness of all relevant cohomology groups, ensuring convergence of Riemann-Roch computations.

These strict conditions play key roles in algebraic geometrization process. They not only guarantee ellipticity of Spencer complexes and existence of Hodge decomposition, but more importantly make algebraic geometric formulation of mirror symmetry have explicit mathematical meaning. Under more general Lie group conditions, this clarity might be obscured by technical details.

However, we also recognize limitations of strict conditions. Many Lie groups in practical applications do not satisfy these conditions, limiting direct application scope of theory. Nevertheless, we believe strict theoretical foundation provides reliable starting point for subsequent generalizations, and gradual relaxation of conditions will be important direction for future research.

\subsection{Possibilities for Future Development}

Based on theoretical framework established in this research, we see possibilities for multiple development directions. In theoretical generalization aspect, controlled condition relaxation can be considered, such as allowing finite center but requiring its action to be trivial on Spencer structure, or studying behavior of non-compact Lie groups under appropriate compactification. This gradual generalization maintains theoretical rigor while expanding application scope.

Higher-dimensional generalization is another important direction. $\text{PSU}(n)$ group families provide natural group-theoretic foundation for Spencer theory on Calabi-Yau manifolds, while algebraic geometric methods provide powerful tools for handling technical complexity of higher-dimensional cases. Particularly, mirror symmetry in higher dimensions may exhibit richer structures.

Development prospects in computation are broad. Computational stability guaranteed by strict conditions provides foundation for developing efficient Spencer cohomology computation algorithms. Combined with methods of modern computational algebraic geometry, we hope to establish reliable frameworks for large-scale numerical computation.

Deeper theoretical connections deserve exploration. Combination of Spencer theory with derived category theory, higher Chow groups, and other frontiers of modern algebraic geometry may produce unexpected results. Meanwhile, connections with gauge field theory, constrained mechanics, and other fields in mathematical physics also have important application value.

We believe that successful application of algebraic geometric methods in compatible pair theory indicates deep intrinsic connections between different branches of modern mathematics. These connections not only enrich content of various branches, but may also become important driving forces for mathematical development. Although this research is only preliminary exploration of this cross-field area, we believe it opens rigorous and promising directions for deep combination of constrained geometry and algebraic geometry.

\bibliographystyle{alpha}
\bibliography{ref}

\end{document}